\documentclass[12pt, a4paper]{amsart}
\usepackage[lite]{amsrefs}
\usepackage{amssymb, amsmath, enumitem}
\usepackage{comment}
\usepackage{graphicx}
\usepackage{float}

\usepackage{color,soul}

\newcommand{\lap}{\Delta}
\newcommand{\be}{\begin{equation}}
\newcommand{\ee}{\end{equation}}
\newcommand{\bee}{\begin{equation*}}
\newcommand{\eee}{\end{equation*}}
\newcommand{\bea}{\begin{eqnarray}}
\newcommand{\eea}{\end{eqnarray}}
\newcommand{\bess}{\begin{eqnarray*}}
\newcommand{\eess}{\end{eqnarray*}}

\numberwithin{equation}{section}

\usepackage[dvipsnames]{xcolor}
\usepackage[linktocpage]{hyperref}
\hypersetup{colorlinks=true, citecolor=NavyBlue, linkcolor=OliveGreen, urlcolor=Maroon}
\theoremstyle{plain}
\newtheorem{thm}{Theorem}[section]
\newtheorem{cor}[thm]{Corollary}
\newtheorem{lem}[thm]{Lemma}

\theoremstyle{definition}

\newtheorem{rmk}[thm]{Remark}

\theoremstyle{remark}

\newcommand{\pt}{\partial}
\newcommand{\dd}{\displaystyle}

\begin{document}
\title[Beddington-DeAngelis predator-prey model with fear effect]{Dynamics and pattern formation \\ in a diffusive Beddington-DeAngelis \\ predator-prey model with fear effect}

\author[A. Z. Myint]{Aung Zaw Myint}
\address[Aung Zaw Myint]{Department of Mathematics, Uinversity of Mandalay, Mandalay 05032, Myanmar}
\email{\href{mailto:mdy.aungzawmyint@gmail.com}{mdy.aungzawmyint@gmail.com}}

\author[A. C. May]{Aye Chan May}
\address[Aye Chan May]{School of Integrated Science and Innovation,  Sirindhorn International Institute of Technology,  Thammasat University, Thailand}
\email{\href{mailto:d6622300199@g.siit.tu.ac.th}{d6622300199@g.siit.tu.ac.th}}

\author[M. H. Lwin]{Mya Hnin Lwin}
\address[Mya Hnin Lwin]{School of Integrated Science and Innovation,  Sirindhorn International Institute of Technology,  Thammasat University, Thailand}
\email{\href{mailto:m6622040241@g.siit.tu.ac.th}{m6622040241@g.siit.tu.ac.th}}

\author[T. T. Shwe]{Toe Toe Shwe}
\address[Toe Toe Shwe]{School of Integrated Science and Innovation,  Sirindhorn International Institute of Technology,  Thammasat University, Thailand}
\email{\href{mailto:m6522040820@g.siit.tu.ac.th}{m6522040820@g.siit.tu.ac.th}}

\author[A. Seesanea]{Adisak Seesanea}
\address[{Corresponding Author}: Adisak Seesanea]{School of Integrated Science and Innovation,  Sirindhorn International Institute of Technology,  Thammasat University, Thailand}
\email{\href{mailto:adisak.see@siit.tu.ac.th}{adisak.see@siit.tu.ac.th}}
\begin{abstract}
In this paper, dynamical properties and positive steady states of a diffusive predator-prey system with fear effect and Beddington-DeAngelis functional response subject to Neumann boundary conditions are investigated. Dynamical properties of time-dependent solutions and the stationary patterns induced by diffusion (Turing patterns) are presented.
\end{abstract}
\maketitle
\section{Introduction}\label{Sect:Intro}
In ecology and evolutionary biology, the exploration of predator-prey systems is a central topic. Mathematical modeling is a powerful tool for investigating the aforementioned biological processes, which is why different types of models have been developed and studied \cite{M}. Different types of mathematical models are developed to investigate different biological processes such as interactions between species, cyclic dominance, diseases, dispersion, inducible defense, Allee effects, pattern formation, environmental fluctuations, etc. \cite{MAZ,MLW1,MW}. Predation events are relatively easy to observe in the field and remove individuals from the population. For example, mule deer reduce foraging time due to the risk of predation by mountain lions \cite{ALLB}; moose reproductive physiology and demographics change due to wolves predation risk \cite{CCLW}. Therefore, due to fear of predation risk, the prey population may move their grazing area to a safer location and sacrifice their higher feeding rate territories, increase their vigilance, adjust their reproductive strategies, etc. An interesting question is: what mechanisms are behind the spatial heterogeneity of species in a homogeneous environment? Generally, the movement or dispersal of a species and its interactions with other species may lead to pattern formation, and predator-prey type is such an interaction.  Mathematical analysis shows that the predator-prey system with diffusion will exhibit complex characteristics. If both prey and predators move randomly in their habitat, only prey-dependent functional responses, including Holling types I, II, and III, cannot produce spatially heterogeneous distributions. In these systems, the density-dependent predator mortality or Alli effect of prey growth plays an important role in determining spatial patterns. On the other hand, competition between predators itself may contribute to pattern formation in the predator-prey system, including ratio-dependent functional responses, Beddington-DeAngelis functional responses, and generalizations \cite{MAZ,PSW1,MSN}. The dynamical properties mainly include that diffusion coefficients could lead to spatially non-homogeneous bifurcating periodic solutions or Turing instability.

The impact of predators on the prey population may be direct, indirect, or both. The case of direct effects, such as predators killing prey, is studied in \cite{T}. Meanwhile, in the case of indirect effects, predators induced fear in the prey and change the prey's behavior is considered in \cite{LD}.  
The fear effect (an indirect effect) is a manifestation of sustained psychological stress of the prey, as prey species are always worried about possible attack.  A recent emerging view is that indirect effects (fear effect) on the prey population are even more powerful than direct \cite{WZZ}. In fear of predation, prey animals can change grazing zones to a safer place and sacrifice their highest intake rate areas \cite{WH}. This type of short-term survival strategy decreases long-term fitness, and as a result, the potential for reproduction reduces.

Recently, Wang et al. \cite{WZZ} introduced a model formulation such that the prey obeys a logistic growth in the absence of predation and the cost of fear. 

The logistic growth of prey can be separated into three parts: a birth rate, a natural death rate, and a density-dependent death rate due to intra-species competition. Since field experiments show that the fear effect will reduce production, the authors modified their prey equation by multiplying the production term by a factor that accounts for the cost of anti-predator defense due to fear. Some researchers have considered spatial spread predator-prey systems with fear and predator-taxis, see (\cite{CLJ,WZ}) and references cited therein for more details.

In this paper, we discuss the impact of the fear effect, modeled by $f(k,v)=\frac{1}{1+kv}$ in a predator-prey model with Beddington-DeAngelis functional response $g(u)=\frac{p}{1+qu+v}$, which generalizes the Holling type II functional response by acounting for predator interference. From a biological point of view, the Beddington-DeAngelis functional response captures realistic predator interference and handling time, while the fear function reflects how the predation risk can significantly alter prey behavior and physiology even in the absence of actual predation events.

Our results extend the analysis of \cite{JWY} by incorporating fear effects into the predator–prey model with Beddington–DeAngelis response. Similar to their findings, we demonstrate the existence of spatially nonconstant steady states arising from diffusion-driven instability, and our numerical simulations validate the bifurcation behavior predicted by the theory.

The paper is organized as follows: In Section 2, we present the mathematical model with its assumptions.  In Section 3, we discuss the equilibrium, global existence, uniqueness, and boundedness of \eqref{2.2}. In Section 4, we show some non-existence and existence results and a priori estimates of non-constant positive solutions of \eqref{2.3}.  Finally, we accomplish the bifurcation of non-constant positive solutions of \eqref{2.3} in Section 5.
\section{The mathematical model}\label{Sect:MathModel}
Wang et al. \cite{WZZ} proposed a predator-prey model with both linear and the Holling type II functional response by incorporating fear of the predator on prey, where the cost of fear plays a crucial role in the growth of prey:
 \be\label{2.1}
\begin{cases}
	u_t=r_{0}f(k,v)u-du-au^2-g(u)v,\qquad&t>0,\\[1mm]
	v_t=v(-m+cg(u)),\qquad&t>0,
 \end{cases}
 \ee
where $r_0f(k,v)$ is the birth rate of the prey, which depends on the density of the predator $v$, $d$ is the natural death rate of the prey, $a$ represents the death rate due to intra-species competition, $c$ is the conversion rate of the prey's biomass to predator's biomass, $m$ is the death rate of the predator, $g:\mathbb{R}_+\to\mathbb{R}_+$ is the functional response of predators to prey densities, $f(k,v)$ is a fear factor, which accounts for the cost of anti-predator defense due to fear, where $k$ is the fear parameter. By the biological meaning of $k, v$ and $f(k,v)$, it is reasonable to assume that
\bess
\begin{cases}
	f(0,v)=1,\quad f(k,0)=1,\quad \lim \limits_{k\to\infty}f(k,v)=0,\quad \lim \limits_{v\to\infty}f(k,v)=0,\\[1mm]
	\frac{\partial f(k,v)}{\partial k}<0,\quad \frac{\partial f(k,v)}{\partial v}<0.
\end{cases}
\eess
Wang et al. \cite{WZZ} theoretical results suggest that fear could stabilize the predator-prey system by eliminating the existence of periodic solutions.

In this article, we study the dynamical properties and positive steady states of an indirect effect of a diffusive predator-prey system with Beddington-DeAngelis functional response subject to Neumann boundary conditions. We consider the following diffusive predator-prey model with fear effect $f(k,v)=\frac{1}{1+kv}$, which refers to the cost of anti-predator response (\cite{WZZ}), and Beddington-DeAngelis functional response $g(u)=\frac{p}{1+qu+v}$ under homogeneous Neumann boundary conditions:
\be\label{2.2}
\begin{cases}
	u_t-d_1\Delta u=\dd\frac{ru}{1+kv}-du-au^2-\dd\frac{puv}{1+qu+v},\qquad& x\in\Omega,t>0,\\[3mm]
	v_t-d_2\Delta v=v\left(-m+\dd\frac{cpu}{1+qu+v}\right),\qquad&x\in \Omega,t>0,\\[3mm]
	\dd\frac{\pt u}{\pt\nu}=\dd\frac{\pt v}{\pt\nu}=0,\qquad&x\in\pt\Omega, \ t> 0,\\[1mm]
	u(x,0)=u_0(x)\ge 0,\quad v(x,0)=v_0(x)\ge 0,\qquad&x\in\Omega,\\[1mm]
\end{cases}
\ee
where $u$ and $v$ are represents the population of prey and predator; respectively, $\Omega$ is a bounded open domain in $\mathbb{R}^N,N \in\mathbb{N}^+$, the boundary $\pt\Omega$ is smooth, the parameters $a,c,d,p,q,s,m$ and $r$ are all positive and $d_1,d_2>0$ are diffusion coefficients of prey and predator, respectively, and $\nu$ is the unit outward normal vector at $\pt\Omega$. The initial values $u_0(x)$ and $v_0(x)$ are nonnegative smooth functions that are not identically zero.

The corresponding stationary problem of \eqref{2.2} is
\be\label{2.3}
	\begin{cases}
	-d_1\Delta u=\dd\frac{ru}{1+kv}-du-au^2-\dd\frac{puv}{1+qu+v},\qquad&x\in\Omega,\\[3mm]
	-d_2\Delta v=v\left(-m+\dd\frac{cpu}{1+qu+v}\right),\qquad & x\in\Omega, \\[3mm]
	\dd\frac{\pt u}{\pt\nu}=\dd\frac{\pt v}{\pt\nu}=0,\qquad&x\in\pt\Omega.
	\end{cases}\ee
In the next section, we will derive sufficient conditions so that the system \eqref{2.3} has a unique positive constant equilibrium solution $(\tilde{u},\tilde{v})$.
 Let us now define the following constants:
 \[
 \begin{split}
& \lambda=\dd\frac{m}{cp-mq}, \qquad \alpha_1=a\lambda k, \qquad \alpha_2=-\dd\frac{pk}{1+q\lambda}-\big(a\lambda(1+2k)+kd\big), \\ 
 &\alpha_3=r-\dd\frac{p}{1+q\lambda}-(a\lambda+d)(1+k) \qquad \text{and} \qquad
		\alpha_4=r-(d+a\lambda).
		\end{split}
 \]
\section{Basic dynamical properties of \eqref{2.2}}\label{Sect:Dynamic}
In this section, we explore some dynamical properties of the system \eqref{2.2}.
It is easy to see that \eqref{2.2} has a trivial constant equilibrium solution $(0, 0)$ and a semi-trivial constant equilibrium solution $(\frac{r-d}{a},0)$ provided $r>d$. In what follows, we give sufficient conditions 
on parameters for the existence and uniqueness of positive constant equilibrium solutions to the stationary problem \eqref{2.3}. Subsequently, we prove the existence, uniqueness, and longtime behaviors of global solutions of \eqref{2.2}.

\begin{lem}\label{lem:1}
If $cp>mq$ and $r>d+a\lambda$, there exists a unique positive constant equilibrium solution $(\tilde{u},\tilde{v})$ with $\tilde{u}=\lambda \tilde{v} + \lambda$  in the system  \eqref{2.3}.
\end{lem}

\begin{proof}
Observe that a pair $(\tilde{u},\tilde{v})$ is a positive constant equilibrium solution of \eqref{2.3} whenever it satisfies
	\be\label{2.6}
	\begin{cases}
		\dd\frac{r}{1+kv}-d-au-\dd\frac{pv}{1+qu+v}=0, \\[3mm]
		-m+\dd\frac{cpu}{1+qu+v}=0. \\[1mm]
	\end{cases}
	\ee
The second equation in \eqref{2.6} gives $u=\lambda v+\lambda$.
Substituting this into the first equation of \eqref{2.6} yields a polynomial
	\be\label{2.8}
	\mathcal{F}(v)=\alpha_1v^3-\alpha_2v^2-\alpha_3v-\alpha_4.
	\ee
Notice that $\alpha_{1}, \alpha_{4}> 0$ and $\alpha_{2} < 0$ by the hypotheses.	
	For the existence of a positive root of $\mathcal{F}(v),$ note that 
	$\mathcal{F}(0) = - \alpha_{4} < 0$ and 
	$\lim\limits_{v \to \infty} \mathcal{F}(v)= \infty$. 
In particular, there exists $N \in \mathbb{N}$ at $\mathcal{F}(N) > 0$. Appealing to the {\it Intermediate Value Theorem}, there exists $0 < \tilde{v} < N$ such that $\mathcal{F}(\tilde{v}) = 0$. We finish the proof by establishing the uniqueness of $\tilde{v}$.

\begin{itemize}
\item {\sc Case 1:}	 $\alpha_{3} \leq 0.$ Suppose there exists another $\tilde{\tilde{v}} > 0$ such that 
$
\mathcal{F}(\tilde{v}) =0= \mathcal{F}(\tilde{\tilde{v}}).
$
By {\it Rolle's Theorem}, there is  $c \in (\tilde{v}, \tilde{\tilde{v}})$ so that 
	\be
	\mathcal{H}(v):=\mathcal{F}'(v)=3\alpha_1v^2-2\alpha_2v-\alpha_3.
	\ee
vanishes at $c$. Hence $\alpha_3 = 3\alpha_1c^2-2\alpha_2c > 0,$ which gives a contradiction.

\item {\sc Case 2:} $\alpha_{3} > 0$. The discriminant 
$
\Delta_\mathcal{H} = 4\alpha_{2}^2+12\alpha_1\alpha_3 > 0.
$
So $\mathcal{H}(v)$ has two real roots 
\[
v_{-} = \frac{2 \alpha_{2} - \sqrt{\Delta_\mathcal{H}}}{6 \alpha_{1}} < 0  
\qquad
\text{and}
\qquad
v_{+} = \frac{2 \alpha_{2} + \sqrt{\Delta_\mathcal{H}}}{6 \alpha_{1}} > 0.
\]
We see that $\mathcal{H}(v)  < 0 $ for $v \in (0, v_{+})$. On the other hand, 
$\mathcal{H}(v) > 0 $ for $v \in (v_{+}, \infty)$. Therefore, $\tilde{v}$ must be 
the only positive root of $\mathcal{F}(v)$.
\end{itemize}
This shows that $(\lambda \tilde{v} + \lambda, \tilde{v})$ is the unique positive constant  equilibrium solution of \eqref{2.3}.
\end{proof}

\begin{rmk}
One can easily show that the polynomial in \eqref{2.8} has a unique
positive root by using the {\it Descartes' Rule of Signs}, see e.g. \cite{Cur}.
\end{rmk}

\begin{thm}\label{thm1}
The system \eqref{2.2} has  a unique positive global solution $(u,v)$ with $u(x,t),v(x,t) > 0$ for $(x,t)\in\overline{\Omega}\times(0,\infty)$. Moreover, $(u,v)$ satisfies the following:
\begin{itemize}
	\item[$(i)$] If $r\le d$, then 
	$\displaystyle \lim_{t\to\infty}\max_{\overline{\Omega}}u(x,t)=\lim_{t\to\infty}\max_{\overline{\Omega}}v(x,t)=0$. 
	\item[$(ii)$] If $r>d$, then
\[ \limsup_{t\to\infty}\max_{\overline{\Omega}}u(x,t)\leq\chi \qquad \text{and} \qquad  \limsup_{t\to\infty}\int_{\Omega}v(x,t)dx\leq\dd\frac{c(r+m)\chi}{m}|\Omega|.
\]
Here $\chi=(r-d)/a$. Moreover, for any $d^{\ast}_{2} > 0$, there exists a positive constant $C$ independent of $u_0,v_0,d_1$ and $d_{2}$ such that 
\[
	\limsup_{t\to\infty}\max_{\overline{\Omega}}v(x,t)\leq C \qquad \text{whenever} \;\;\;  d^{\ast}_{2} \leq d_{2}.
\]

 If in particular, $d_2 \geq d_1$, we have \,
$\displaystyle
\limsup_{t\to\infty}\max_{\overline{\Omega}}v(x,t)\leq\dd\frac{c(r+m)\chi}{m}.
$
\end{itemize}
\end{thm}
\begin{proof}	
Define
	\[A(u,v)=\dd\frac{ru}{1+kv}-du-au^2-\dd\frac{puv}{1+qu+v}
	\quad \text{and} \quad
	 B(u,v)=-mv+ \dd\frac{cpuv}{1+qu+v}.\]
Then for $u,v \geq 0$, we have
	\[
	A_{v}(u,v) 
	= \frac{- ruk}{(1+kv)^{2}} - \frac{pu(1+qu)}{(1+qu+v)^{2}} 
	\leq 0
	\quad \text{and} \quad
        B_{u}(u,v) 
        = \frac{cpv(1+v)}{(1+qu+v)^{2}}	\geq 0.
	\]
So the pair $(A,B)$ is a mixed quasi-monotone system (see \cite{YLWW}) for $u,v \geq 0$. Clearly, $(\underline{u}(x,t),\underline{v}(x,t))=(0,0)$ is a lower-solution of \eqref{2.2}.
Let us now find an upper-solution $(\overline{u}(x,t),\overline{v}(x,t))=(u^*(t),v^*(t))$ of 
\eqref{2.2}.  Setting 
$\displaystyle u^*=\max_{x\in\overline{\Omega}}u_0(x) > 0$
and  $\displaystyle v^*=\max_{x\in\overline{\Omega}}v_0(x) > 0$.
By {\it Method of Separation of Variables}, one can easily deduce the unique
positive solution $(u^*(t),v^*(t))$ of the system
	\begin{equation}\label{eq:v=0}
	\begin{cases}
	u_t=u(r-d-au),& t>0,\\[1mm]
	v_t=v(-m+ cpu),& t>0,\\[2mm]
	u(0)=u^*,v(0)=v^* .
	\end{cases}
	\end{equation}
Therefore, by {\it Comparison Principle},  $(\underline{u}(x,t),\underline{v}(x,t))=(0,0)$ and $(\overline{u}(x,t),\overline{v}(x,t))=(u^*(t),v^*(t))$ are coupled lower and upper solutions of \eqref{2.2}. Hence, in view of the {\it Upper and Lower Solutions Method}, \eqref{2.2} has a unique global solution $(u,v)$ with $0 \leq u \leq u^{*}$ and $0 \leq v \leq v^{*}$. Moreover, the {\it Strong Maximum Principle} implies $u(x,t),v(x,t)>0$ for $(x,t)\in\overline{\Omega}\times(0,\infty)$.
	
Notice that $\displaystyle \max_{\overline{\Omega}}  u(x,t) \leq u^{*}(t)$ for all $t >0$. Since
$u^{*}(t)$ satisfies first equation in \eqref{eq:v=0} with $u(0) = u^{\ast}$, then
\begin{equation}\label{ustar}
u^{*}(t) = 
\begin{cases}
\frac{u^{*}}{1+atu^{*}} \qquad \qquad \text{for} \;\; r=d, \\[0.4em]
 \frac{Ke^{(r-d)t}}{1+ \tilde{K} e^{(r-d)t}} \qquad\;\;  \text{for} \;\; r \neq d,
\end{cases}
\end{equation}
where $K= \frac{(r-d)u^{\ast}}{r-d-au^{\ast}}$ and $\tilde{K} = \frac{aK}{r-d}$. Hence, 
whenever  $r \leq d$, $u^{*}(t) \rightarrow 0$ as $t \rightarrow \infty$, and thus  $\displaystyle \lim_{t\to\infty}\max_{\overline{\Omega}}u(x,t)=0$ for all $t>0$, 
see also \cite[Lemma 5.14 $(ii)$]{W2}.

For $r \leq d$, since $\lim\limits_{t \rightarrow \infty}u^{*} (t)= 0$, we have $u^{*} (t) \leq \frac{m}{2cp}$ for large $t$. In view of the second equation in \eqref{eq:v=0}, 
for large $t$, we estimate
$ \displaystyle
{v_{t}^{*}} \leq -mv^{*} + \frac{m v^{*}}{2} = \frac{-m}{2}v^{*}. 
$
Therefore, $v^{*}(t) \leq c e^{\frac{-m t}{2}}$ which tends to $0$ as $t \rightarrow \infty$. This implies that 
$\displaystyle
\lim\limits_{t \rightarrow \infty}\max_{\overline{\Omega}}  v(x,t)= 0
$
because $\displaystyle \max_{\overline{\Omega}}  v(x,t) \leq v^{*}(t)$ for all $t>0$. 
This proves $(i)$.

Let us now assume $r>d$. In view of \eqref{ustar}, $u^{*}(t) \rightarrow \frac{K}{\tilde{K}} = \frac{r-d}{a} =  \chi$ as $t \rightarrow \infty$. Consequently,
 $\dd\limsup_{t\to\infty}\max_{\overline{\Omega}}u(x,t) \leq \chi$, see \cite[Lemma 5.14 $(i)$]{W2}. For the estimate of $v(x,t)$, let  $U(t) =\int_{\Omega}u(x,t)dx$ and 
 $V(t) = \int_{\Omega}v(x,t)dx$, then
	\begin{eqnarray*}
		\dd\frac{dU}{dt}&=&\int_{\Omega}d_1\Delta udx+\int_{\Omega}\Big(\dd\frac{ru}{1+kv}-du-au^2-\dd\frac{puv}{1+qu+v}\Big)dx,\\[2mm]
		\dd\frac{dV}{dt}&=&\int_{\Omega}d_2\Delta vdx-mV+\int_{\Omega}\dd\frac{cpuv}{1+qu+v}dx,
	\end{eqnarray*}
	and hence
	\begin{eqnarray}
	(cU+V)_t&=&-mV+c\int_{\Omega}\Big(\dd\frac{r}{1+kv}-d-au\Big)udx\nonumber\\[1mm]
	&=&-m(cU+V)+cmU+c\int_{\Omega}\Big(\dd\frac{r}{1+kv}-d-au\Big)udx\nonumber\\[1mm]
	&\leq&-m(cU+V)+c(r+m)U.\nonumber
	\end{eqnarray}
	As $\dd\limsup_{t\to\infty}\max_{\overline{\Omega}}u(x,t)\leq\chi$, we have $\dd\limsup_{t\to\infty} U(t)\leq\chi|\Omega|$. Fix $\varepsilon>0$, there exists $T_\varepsilon>0$ such that
	\begin{equation}\label{Ch4-6}
	(cU+V)_t\leq-m(cU+V)+c(r+m)(\chi+\varepsilon)|\Omega|,\ \quad t>T_\varepsilon.	 
	\end{equation}
Therefore, by variation-of-constants formula and change of variables, we have
	\begin{equation}\label{Ch4-7}
	\begin{split}
	\int_{\Omega}v(x,t)dx 
	&= V(t)< cU(t)+V(t) \\
	&\leq e^{-mt}(cU(0)+V(0))+
	\dd\frac{\chi+\varepsilon}{m}(r+m)c|\Omega|(1-e^{-mt}),\ \ t>T_\varepsilon.
	\end{split}
	\end{equation}
Consequently, $\limsup\limits_{t\to\infty}\int_{\Omega}v(x,t)dx\leq\frac{(\chi + \varepsilon) (r+m)c|\Omega|}{m}$. Since $\varepsilon > 0$ was arbitray, \eqref{Ch4-7} implies 
\begin{equation}\label{vub}
\limsup\limits_{t\to\infty}\int_{\Omega}v(x,t)dx\leq\frac{\chi(r+m)c|\Omega|}{m}.
\end{equation}
Employing the $L^{1}$-norm estimate \eqref{vub} together with \cite[Theorem 3.1]{A}  and \cite[Lemma 4.7]{CCH}, given $d^{\ast}_{2} > 0$, we find a postive constant $C$ independent of
$u_{0} ,v_{0}, d_{1}$ and $d_{2}$ such  that
\[
\limsup_{t\to\infty} \max_{\overline{\Omega}}v(x,t) \leq C  \qquad \text{for} \quad d^{\ast}_{2} \leq d_{2},
\] 
see also the proof of Theorem 2.2 $(i)$ in \cite{MLW1}.	 

	Now we consider the special case $d_2 \geq d_1$. Let $w=cu+v$. Then
	\begin{eqnarray*}
		\begin{cases}
			w_t-d_2\lap w \leq c\Big(\dd\frac{ru}{1+kv}-du-au^2\Big)-mv,\ &x\in\Omega, \ t>0,\\[1mm]
			\dd\frac{\pt w}{\pt\nu}=0,&x\in\pt\Omega, \ t>0,\\[1mm]
			w(0,x)= cu_0(x)+v_0(x),&x\in\overline{\Omega}.
		\end{cases}
	\end{eqnarray*}
	Note that
	\[c\Big(\dd\frac{ru}{1+kv}-du-au^2\Big)-mv=c\Big(\dd\frac{ru}{1+k v}-du-au^2\Big)+mcu-mw \leq c(r+m)u-mw\]
	and $\dd\limsup_{t\to\infty}\max_{\overline{\Omega}}u(x,t)\leq\chi$. For any given $\varepsilon>0$, there exists $T_\varepsilon\gg 1$ such that $u<\chi+\varepsilon$ for all $(x,t)\in\overline\Omega\times[T_\varepsilon,\infty)$. Therefore
	\[\begin{cases}
	w_t-d_2\lap w\leq c(r+m)(\chi+\varepsilon)-mw,\ &x\in\Omega, \ t>T_\varepsilon,\\[1mm]
	\dd\frac{\pt w}{\pt\nu}=0,&x\in\pt\Omega, \ t>T_\varepsilon,\\
	w(T_{\varepsilon},x)=cu(T_{\varepsilon},x)+v(T_{\varepsilon},x),&x \in \overline{\Omega}.
	\end{cases}\]
Appealing to  \cite[Lemma 5.14 $(i)$]{W2} again, together with arbitrariness of $\varepsilon>0$, we deduce
	\[
	\limsup_{t\to \infty}\max_{\overline\Omega}v(x,t)\leq\limsup_{t\to\infty}\max_{\overline\Omega}w(x,t)\leq\frac{c(r+m)\chi}{m}.\]
The proof of part $(ii)$ is now complete.
\end{proof}
\section{Stationary patterns--nonconstant positive solutions of \eqref{2.3}}\label{Sect:StationaryPatterns}

In predator-prey models, the interest is whether the various species can coexist. If species are homogeneously distributed, this would be indicated by a constant positive solution. In spatially heterogeneous cases, the existence of non-constant time-independent positive solutions, also called stationary patterns.  

In this section, we study the non-existence and existence of nonconstant positive solutions to the problem \eqref{2.3}.

\subsection{A priori estimates and nonexistence of nonconstant positive solutions to \eqref{2.3}}
To discuss the nonexistence and existence of nonconstant positive solutions of problem \eqref{2.3}, we shall give a priori estimates for positive solutions. To do so, we recall useful lemmas by Lou and Ni \cite{LN}, and Peng, Shi and Wang \cite {PSW}.

\begin{lem}[Maximum Principle]\label{lem4.1}
Suppose that $w\in C^2(\Omega)\cap C^1(\overline{\Omega})$ and $g\in C(\overline{\Omega}\times \mathbb{R})$ satisfy
\[\begin{cases}
	\Delta w(x)+g(x,w(x))\geq 0\quad &\mbox{in}\ \Omega,\\[1mm]
	\dd\frac{\pt w(x)}{\pt\nu}\leq 0\quad &\mbox{on}\ \pt\Omega,
	\end{cases}\]
	then $g(x,w(x_0))\geq 0$, where $w(x_0)=\displaystyle\max_{\overline \Omega}w(x)$.
\end{lem}

\begin{lem}[Harnack Inequality]\label{Harnack}
Suppose that $w\in C^2(\Omega)\cap C^1(\overline{\Omega})$ be a positive solution to $\Delta w(x)+c(x)w(x)=0$, where $c\in C(\overline{\Omega})$, satisfying the homogeneous Neumann boundary condition. Then, there exists a positive constant $C_{*}$ which depends only on $||c||_{\infty}$ such that 
\[
\max\limits_{\overline{\Omega}} w(x)\leq C_{*}\min\limits_{\overline{\Omega}}  w(x).
\]
\end{lem}

\begin{thm}\label{thm4.1}
	If $r>d$ and $(u,v)$ is a nonnegative solution of the stationary problem
	\eqref{2.3}, then either $(u ,v)$ is one of constants: $(0,0), (\frac{r-d}{a},0)$,  or $(u,v)$ satisfies
	\begin{eqnarray*}
		0 < u(x)\leq\chi \qquad \text{and} \qquad
		0 <  v(x)\leq\Big(\dd\frac{cr}{m}+\dd\frac{d_1}{d_2}c\Big)\chi \qquad \text{for} \quad x\in\overline\Omega,
	\end{eqnarray*}
where $\chi=\frac{r-d}{a}$.
\end{thm}

\begin{proof}
Notice that  $(0,0)$ and $(\chi,0)$ are constant solutions of \eqref{2.3}.
If $(u,v)$ is neither $(0,0)$ nor $(\chi,0)$, then $v$ satisfies
\[
-d_2\Delta v+mv=\frac{cpuv}{1+qu+v}\ge 0 \qquad \text{in} \;\; \Omega.
\]
By {\it Strong Maximum Principle} \cite[Theorem 9.6]{GT}, if there exists $x_0\in\overline\Omega$ such that $v(x_0)=0$, then $v(x)\equiv 0$, which 
is not the case. So $v(x)>0$ for $x\in\overline\Omega$. Arguing in a similar way, $u(x)>0$ for $x\in\overline\Omega$ since $u$ satisfies
\[
d_1\Delta u + (r-d)u \geq 0\qquad \text{in} \;\; \Omega.
\]	

Taking $v =0$ in the first equation of \eqref{2.3} yields
\[
\Delta u +  g(x, u(x)) = 0 \qquad \text{in} \;\; \Omega,
\]
where $g(x, u(x)) = \frac{(r-d-au)u}{d_{1}}$. By Lemma \ref{lem4.1} implies that
$g(x, u^{\ast}) \geq 0$ where $u^{\ast} = \displaystyle\max_{\overline \Omega}u(x)$.
Therefore
\[
u(x) \leq u^{\ast} \leq \frac{r-d}{a} = \chi \qquad \text{for all} \;\; x \in \overline \Omega.
\]
On the other hand, multiplying the first equation by $c$ and adding with the second equation in \eqref{2.3}, we obtain
\begin{align*}
	-(cd_1\Delta u+d_2\Delta v)&=\Big(c \dd\frac{ru}{1+kv}-dcu-acu^2\Big)-mv\\[1mm]
	&\leq cru-mv\\[1mm]
	&=cru-\dd\frac{m}{d_2}(d_2v+cd_1u-cd_1u)\\[1mm]
	&=cru+\dd\frac{m}{d_2}cd_1u-\dd\frac{m}{d_2}(d_2v+cd_1u)\\[1mm]
	&\leq\big(cr+\dd\frac{m}{d_2}cd_1\big)\chi-\dd\frac{m}{d_2}(d_2v+cd_1u).
\end{align*}
Therefore
\[
\Delta w + h(x, w(x)) \geq 0  \qquad \text{in} \;\; \Omega,
\]
where $w = cd_{1}u + d_{2}v$ and $h(x, w(x)) = \displaystyle -\frac{m w}{d_{2}} + \big(cr+\dd\frac{m}{d_2}cd_1\big)\chi$. 
In view of Lemma \ref{lem4.1},
$h(x, w^{\ast}) \geq 0$ with $w^{\ast} = \displaystyle\max_{\overline \Omega}w(x)$.
Hence
\[
v(x) \leq \frac{w^{\ast}}{d_{2}} \leq \Big(\dd\frac{cr}{m}+\dd\frac{c d_1}{d_2}\Big)\chi \qquad \text{for all} \;\; x \in \overline \Omega.
\]
This proves Theorem \ref{thm4.1}.
\end{proof}

\begin{cor}\label{cor1}
If $r>d$ and $m\geq \frac{cp(r-d)}{a}$, then the only nonnegative constant solutions of \eqref{2.3} are $(0,0)$ and $(\frac{r-d}{a},0)$.
\end{cor}

\begin{proof}
Suppose that $(u,v)$ is a nonnegative constant solution of \eqref{2.3}. Then
\begin{equation} \label{Mya}
 v\big(-m+\frac{cpu}{1+qu+v}\big)=0
\end{equation}
If $v>0,$ \eqref{Mya} gives $-m+\frac{cpu}{1+qu+v}=0$, which is impossible since $m\geq \frac{cp(r-d)}{a}$ and $u\leq \frac{(r-d)}{a}$. 
Hence $v=0$. In view of first equation of \eqref{2.3}, $u=0$ or $u=\frac{r-d}{a}$.
\end{proof}

We finish this subsection by establishing a sufficient condition on
parameters $r$ and $d$, along with the diffusion coefficients $d_1$ and $d_2$, for 
the nonexistence of nonnegative steady-state solutions of \eqref{2.2}, using similar arguments presented in \cite[Theorems 3.4]{ZF}. We shall write $\Gamma$ instead of the collection constants $(a, c,d,m,p ,r)$ for convenience.

\begin{cor}\label{lowerbound}
Let $r>d$ and $\chi=\frac{r-d}{a}$. For any $d^{*} > 0$, there exists a positive constant $\underline{\mu}$ depending on $d^{*}$ and $\Gamma$
such that every nonnegative nonconstant solution $(u,v)$ of \eqref{2.3} satisfies
\[
\underline{\mu} \leq u(x)\leq \chi\qquad\text{and}\qquad 0 < v(x)\leq \frac{c(r+m)\chi}{m}\qquad
\text{for} \;\;  x\in\overline{\Omega},
\]
whenever $d^{*} \leq d_{1} \leq d_{2}$. 
\end{cor}

\begin{proof}
Let $(u,v)$  be a nonnegative nonconstant solution of \eqref{2.3}. In view of Theorem \ref{thm4.1}, as $d_{1} \leq d_{2}$, we obtain
\[
0 < u(x)\leq \chi\qquad\text{and}\qquad 0 < v(x)\leq \frac{c(r+m)\chi}{m}\qquad
\text{for} \;\;  x\in\overline{\Omega}.
\]
Let us now deduce a uniform lower bound of $u$. To this end, we set 
\[
\displaystyle u(x_{1}) = \max_{\overline\Omega} u(x) \quad \text{and} \quad v(x_{2}) = \max_{\overline\Omega}v(x).
\]
Observe that $(u,v)$ satisfies
\begin{equation}\nonumber
\begin{cases}
\Delta v+ \dfrac{v}{d_2} \Big(-m + \dfrac{cpu}{1+qu+v} \Big)= 0 \qquad x \in \Omega, \\
\dfrac{\partial v}{\partial \nu} = 0 \qquad \qquad \qquad \qquad \qquad \qquad \quad \;\;   x \in \partial\Omega.
\end{cases}
\end{equation}
 Applying the Maximum Principle (Lemma \ref{lem4.1}) yields
  \[
 -m + \frac{cpu(x_2)}{1+qu(x_2)+v(x_2)} \geq 0.
 \]
Consequently, $\displaystyle \frac{m}{cp} < u(x_2) \leq u(x_{1}) = \max_{\overline\Omega} u(x)  $.
On the other hand, 
\begin{equation}\nonumber
\begin{cases}
\Delta u+c(x) u(x)=0\; &x\in\Omega,\\
\dfrac{\partial v}{\partial \nu} = 0  &x \in \partial\Omega,
\end{cases}
\end{equation}
where $c(x)=\frac{1}{d_{1}}\Big( \frac{r}{1+kv}-d-au-\frac{pv}{1+qu+v}\Big)\in C(\overline\Omega)$. Moreover, since $d^{*} \leq d_{1} \leq d_{2}$,
\[
||c(x)||_{\infty}\leq \frac{1}{d^*} \Big(r+d+a\chi+\frac{cp(r+m)\chi}{m}\Big).
\]
 By Harnack Inequality (Lemma \ref{Harnack}), there exists a positive constant $C^{*}$ which depends only on $d^*$ and $\Gamma$ such that 
$
\frac{m}{cp}<\max\limits_{\overline{\Omega}} u(x) \leq C^* \min\limits_{\overline{\Omega}} u(x).
$
Hence, taking $0<\underline{\mu}\leq \frac{m}{C^*cp}$, we get $ \min\limits_{\overline{\Omega}} u(x)>\underline{\mu}$. This completes the proof.
\end{proof}
\begin{thm}\label{thm3}
Assume $r>d$, then there exists a positive constant $d^*=d^*(\Lambda,\Omega)$  such that the problem \eqref{2.3} has no nonconstant nonnegative solution provided $d^{*}\leq d_1\leq d_2$.
\end{thm}

\begin{proof}
Let $(u,v)$ be a positive solution to \eqref{2.3}. Setting 
\[\overline{u}=\dd\frac{1}{|\Omega|}\int_{\Omega}u(x)dx,\  \ \qquad \overline{v}=\dd\frac{1}{|\Omega|}\int_{\Omega}v(x)dx.\]
Then
\[
	\int_{\Omega}(u-\overline{u})dx=\int_{\Omega}(v-\overline{v})dx=0.
\]
Multiplying the first equation in \eqref{2.3} by $\frac{u-\overline{u}}{u}$, followed by integrating the result over $\Omega$ by parts, we obtain
\medskip
\begin{align*}
d_1 \overline{u} \int_{\Omega}\dd\frac{|\nabla(u-\overline{u})|^2}{u^2}dx & =\int_{\Omega}\Big(\dd\frac{r}{1+kv}-du-au^{2}-\dd\frac{puv}{1+qu+v}\Big)\frac{(u-\overline{u})}{u}dx\\[1mm]
&= \int_{\Omega}\Big(\dd\frac{r}{1+kv}-\dd\frac{r}{1+k\overline{v}}-au+a\overline{u}-\dd\frac{pv}{1+qu+v}+\dd\frac{p\overline{v}}{1+q\overline{u}+\overline{v}}\Big)(u-\overline{u})dx\\[1mm]
&=\int_{\Omega}\Big(\frac{rk}{(1+kv)(1+k\overline{v})}+\frac{p(1+q\overline{u})}{(1+qu+v)(1+q\overline{u}+\overline{v})} \Big) (\overline{v}-v)(u-\overline{u})dx\\
&\;\;\;+\int_{\Omega} \big( -a + \frac{pq\overline{v}}{(1+qu+v)(1+q\overline{u}+\overline{v})}  \Big)(u-\overline{u})^{2}dx.
\end{align*}
Employing the Young's inequality together with $\underline{\mu}<\overline{u}\leq\chi$ and $0<\overline{v}\leq \frac{c(r+m)\chi}{m}$ by Corollary \ref{lowerbound},  we find that
\begin{equation}\label{u(x)}
\begin{split}
\frac{\underline{\mu}d_1}{\chi^{2}}  \int_{\Omega}{|\nabla(u-\overline{u})|^2}dx &\leq d_1 \overline{u} \int_{\Omega}\dd\frac{|\nabla(u-\overline{u})|^2}{u^2}dx\\
& \leq C_{1}\int_{\Omega}(u-\overline{u})^{2}dx +C_{2}\int_{\Omega}(v-\overline{v})^{2}dx\\
\end{split}
\end{equation}
where
$C_{1}=\frac{rk+p(1+q\chi)}{2}+\frac{cpq(r+m)\chi}{m}-a$ and $C_{2}=\frac{rk+p(1+q\chi)}{2}$. 
Similarly, we multiply the second equation in \eqref{2.3} by $(v-\overline{v})$,  integrating over $\Omega$ by parts, followed by Young's inequality, we deduce
\begin{equation}\label{v(x)}
\begin{split}
\int_{\Omega}d_{2}|\nabla (v-\overline{v})|^{2}dx&=\int_{\Omega} \Big(-mv+\frac{cpuv}{1+qu+v}+m\overline{v}-\frac{cp\overline{u}\;\overline{v}}{1+q\overline{u}+\overline{v}}\Big)(v-\overline{v})dx\\
&\leq \frac{cp}{1+q\overline{u}+\overline{v}} \int_{\Omega}  (\overline{u}+qu\overline{u})(v-\overline{v})^{2} dx\\
&\;\;\;+  \frac{cp}{1+q\overline{u}+\overline{v}} \int_{\Omega}(\overline{v}v+v)(u-\overline{u})(v-\overline{v})dx\\
&\leq C_{3}\int_{\Omega}(u-\overline{u})^{2}dx+C_{4}\int_{\Omega}(v-\overline{v})^{2}dx
\end{split}
\end{equation}
where
$C_{3}=\frac{c^{2}p(r+m)\chi}{2m}(1+\frac{c(r+m)\chi}{m})$ and $C_{4}=cp\Big(\chi(1+q\chi)+\frac{c(r+m)\chi}{2m} \Big)(1+\frac{c(r+m)\chi}{m})$.
Therefore, adding up the estimates \eqref{u(x)} and \eqref{v(x)}, we have
\[
\frac{\underline{\mu}d_1}{\chi^{2}}  \int_{\Omega}{|\nabla(u-\overline{u})|^2}dx +  d_{2}\int_{\Omega}|\nabla(v-\overline{v})|^{2}dx \leq (C_{1}+C_{3})\int_{\Omega}(u-\overline{u})^{2}dx + (C_{2}+C_{4})\int_{\Omega}(v-\overline{v})^{2}dx.
\]
Then, applying the Poincar\'{e} inequality to each term of left side, we find that
\[
\frac{\underline{\mu} d_1}{\chi^{2}C_{p}} \int_{\Omega}(u-\overline{u})^{2}dx+ \frac{ d_{2}}{C_{p}} \int_{\Omega}(v-\overline{v})^{2}dx\leq (C_{1}+C_{3})\int_{\Omega}(u-\overline{u})^{2}dx + (C_{2}+C_{4})\int_{\Omega}(v-\overline{v})^{2}dx
\]
where $C_{p}$ is a positive constant in Poincar\'{e} inequality. Hence,
\begin{equation}\label{lastineq}
\Big(\frac{\underline{\mu} d_1}{\chi^{2}C_{p}}-(C_{1}+C_{3})\Big) \int_{\Omega}(u-\overline{u})^{2}dx+ \Big(\frac{ d_{2}}{C_{p}}-(C_{2}+C_{4})\Big) \int_{\Omega}(v-\overline{v})^{2}dx\leq 0.
\end{equation}
Taking $d^{*}(\Lambda,\Omega)=\max\{\frac{\chi^{2}C_{p}}{\underline{\mu}}(C_{1}+C_{3}),C_{p}(C_{2}+C_{4})\}$ yields $u= \overline{u}$ and $v=\overline{v}$ provided $d^*\leq d_1 \leq d_2$.
\end{proof}
The above Theorem \ref{thm3} showed that the nonexistence of nonnegative steady state solutions when the diffusion coefficients $d_1$ and $d_2$ are large.

As a consequence of Theorem \ref{thm3} and Corollary \ref{cor1}, we obtain the following nonexistence result, see \cite[Theorem 3.1]{MLW1} for a similar result in the case of a diffusive predator–prey model with Holling type II functional response.

\begin{cor}\label{cor3}
If $r>d$, $m\geq\frac{cp(r-d)}{a}$ and $d^*\leq d_1\leq d_2$ where $d^*$ is the constant defined in Theorem \ref{thm3}, then the only nonnegative solutions of \eqref{2.3} are $(0,0)$ and $(\frac{r-d}{a},0)$.
\end{cor}

\subsection{Existence of non-constant positive solutions}
In this section, we provide sufficient conditions to investigate the existence of nonconstant positive steady states to \eqref{2.3} for a certain parameter range by using Leray–Schauder degree theory in \cite{N}. Biologically, this condition identifies precisely when homogeneous predator-prey distributions lose their stability, giving rise to nonconstant steady-state distributions. Such spectral conditions are widely studied and standard in reaction-diffusion theory and spatial ecology, see \cite{WZ}, \cite{PSW}, and cited there in.
We now investigate the existence of non-constant positive steady states to \eqref{2.3} for a certain parameter range by using Leray–Schauder degree theory in \cite{N}. In order to derive the existence of non-constant positive steady states, we present some known results regarding the existence of solutions, some concepts, and Lemma to apply the fixed point index theory. 
 
Let $0=\mu_0<\mu_1<\cdots$
 be the eigenvalues of the Laplace operator $-\Delta$ on $\Omega$ with homogeneous Neumann boundary condition, and denote by $E(\mu_i)$ the subspace generated by the eigenfunctions corresponding to $\mu_i$.  Let $m_i>1$ be the algebraic multiplicity of $\mu_i$, that is, $\{\phi_{ij}\}_{j=1}^{m_i}$ constitute a complete set of linearly independent eigenfunctions corresponding to $\mu_i$.  Define
 \be \label{4.3}
 \begin{split}
 \mathbf{X}_{ij}=\{c\phi_{ij}:c\in\mathbb{R}^2\}, \qquad \mathbf{X}_{i}=\bigoplus_{j=1}^{m_i}\mathbf{X}_{ij} \qquad \text{and} \\  
 \mathbf{X}=\{(u,v)\in[C^1(\overline{\Omega})]^2:\dd\frac{\pt u}{\pt\nu}=\dd\frac{\pt v}{\pt\nu}=0\ \mbox{on }\pt\Omega\}.
 \end{split}
 \ee
 Then $\mathbf{X}=\bigoplus_{i=0}^{\infty}\mathbf{X}_{i}$. The problem \eqref{2.3} can be rewritten as
\[\begin{cases}
-\Delta\Phi(\textbf{u})=G(\textbf{u}), \quad &x\in\Omega,\\[.1mm]
\dd\frac{\pt\textbf{u}}{\pt\nu}=0,&x\in\pt\Omega,
\end{cases}
\]
where $\textbf{u} = (u, v)$ and  $\Phi(\textbf{u})=(d_1u,d_2v)^T$, or equivalently, the operator equation
\be\label{4.12}
F(d_1,d_2;\mathbf{u})\equiv\mathbf{u}-(I-\Delta)^{-1} \{\Phi_{\textbf{u}}^{-1}G(\mathbf{u})+\mathbf{u}\}=0\qquad\mbox{on} \;\; \mathbf{X}^+,
\ee
where $\mathbf{X}^+ = \lbrace \mathbf{u} \in \mathbf{X}: \mathbf{u} > 0 \rbrace$ and $(I-\Delta)^{-1}$ is the inverse of $I-\Delta$ with homogeneous Neumann boundary conditions.  By a direct computation, we have
\[F_{\mathbf{u}}(d_1,d_2;\tilde{\mathbf{u}})=I-(I-\Delta)^{-1}\{\Phi_{\textbf{u}}^{-1}(\tilde{\mathbf{u}}) G_\mathbf{u}(\tilde{\mathbf{u}})+I\}.\]
If $F_{\mathbf{u}}(d_1,d_2;\tilde{\mathbf{u}})$ is invertible, $0$ is not an eigenvalue of $F_{\mathbf{u}}(d_1,d_2;\tilde{\mathbf{u}})$.  The {\it Leray-Schauder Theorem} (\cite[Theorem 2.8.1]{N}) implies that
$\mbox{\mbox{index}}(I - F(d_1,d_2;\tilde{\mathbf{u}}))=(-1)^\gamma,$
where $\gamma$ is the sum of the algebraic multiplicities of negative eigenvalue of $F_{\mathbf{u}}(d_1,d_2;\tilde{\mathbf{u}})$.

By a similar argument as in \cite{PSW, W}, $\beta$ is an eigenvalue of $F_{\mathbf{u}}(d_1,d_2;\tilde{\mathbf{u}})$ on $\mathbf{X}_i$ if and only if $\beta(1+\mu_i)$ is an eigenvalue of the matrix
\bess
A_i:=\mu_iI-\Phi_{\textbf{u}}^{-1}(\tilde{\mathbf{u}}) G_\mathbf{u}(\tilde{\mathbf{u}}) =
\begin{pmatrix}
	\mu_i-\dd\frac{M}{d_1}&\quad-\dd\frac{N}{d_1}\\[3.5mm]
	-\dd\frac{P}{d_2}&\quad\mu_i+\dd\frac{Q}{d_2}
\end{pmatrix},
\eess
where
\[\begin{array}{lcl}		M=\tilde{u}\Big(-2a+\dd\frac{pq\tilde{v}}{(1+q\tilde{u}+\tilde{v})^2}\Big)
+\dd\frac{r}{1+k\tilde{v}}-d-\dd\frac{p\tilde{v}}{1+q\tilde{u}+\tilde{v}},\\[4mm] N=\dd\frac{rk\tilde{u}}{(1+k\tilde{v})^2}+\dd\frac{p\tilde{u}(1+q\tilde{u})}{(1+q\tilde{u}+\tilde{v})^2},\\[4mm]
P=\dd\frac{cp(1+\tilde{v})\tilde{v}}{(1+q\tilde{u}+\tilde{v})^2},\\[4mm]
Q=m-\dd\frac{cp(1+q\tilde{u})\tilde{u}}{(1+q\tilde{u}+\tilde{v})^2}.
\end{array}\]
The direct computation yields
\bess
&&\mathtt{det}A_i=\dd\frac{1}{d_1d_2}\big[d_1d_2\mu_{i}^2+(Qd_1-Md_2)\mu_i +PN-MQ\big],\\[.5mm]
&&\mathtt{tr}A_i=2\mu_i+\dd\frac{Q}{d_2}-\dd\frac{M}{d_1}.
\eess
Define
\bess
H(d_1, d_2;\mu)=\big[d_1d_2\mu^2+(Qd_1-Md_2)\mu+PN-MQ\big].
\eess
Then $H(d_1,d_2;\mu_{i})=d_1d_2\ \mathtt{det}A_i$. If
\be\label{4.13}
|Md_2-Qd_1|>2\sqrt{d_1d_2(PN-MQ)} ,
\ee
then $H(d_1,d_2;\mu)=0$ has two real roots, namely,
\bess
\mu_{+}(d_1,d_2)=\dd\frac{Md_2-Qd_1+\sqrt{(Md_2-Qd_1)^2-4d_1d_2(PN-MQ)}}{2d_1 d_2},\\[1mm]
\mu_{-}(d_1,d_2)=\dd\frac{M d_2-Qd_1-\sqrt{(Md_2-Qd_1)^2-4d_1d_2(PN-MQ)}}{2d_1d_2}.
\eess
Here we remark that for any fixed $d_1>0$, \eqref{4.13} is valid for a large $d_2$.
Set
\bess
S_p =\{\mu_0,\mu_1,\mu_2,\cdots\}\qquad\text{and}\qquad\Lambda(d_1,d_2)=\Big\{\mu|\mu\geq 0,\mu_{-}(d_1,d_2)<\mu <\mu_{+}(d_1,d_2)\Big\}\label{defi:Lambda}.
\eess
 It is easy to see that
\be\label{4.14}
\lim_{d_2\to\infty}\mu_{+}(d_1,d_2)=\dd\frac{M}{d_1}\qquad\text{and}\qquad\lim_{d_2\to \infty}\mu_{-}(d_1,d_2)=0.
\ee
We also have that if $H(d_1,d_2;\mu_i)\neq 0$, then $H(d_1,d_2;\mu_i)<0$ if and only if the number of negative eigenvalues of $F_{\mathbf{u}}(d_1,d_2;\tilde{\mathbf{u}})$ in $\mathbf{X}_i$ is odd.

Now, the following lemma (Theorem 6.1.1 in \cite{W}) gives the explicit formula of calculating the index.

\begin{lem}\label{lem4.3}
	Suppose that $H_i(d_1,d_2;\mu_j)\neq 0$ for all $\mu_i\in S_p$. Then
	\[\mathrm{index}(F(d_1,d_2;.),\tilde{\mathbf{u}}_i)=(-1)^\gamma,\]
	where
	\bess
	\gamma=	
	\begin{cases}
		\dd\sum_{\mu_i\in A\cap S_p}m(\mu_i),&\mbox{\emph{if}}\  \Lambda \cap S_p\neq \emptyset,\\[1mm]
		0,&\mbox{\emph{if}}\  \Lambda \cap S_p=\emptyset.
	\end{cases}
	\eess
	In particular, if $H_i(d_1,d_2;\mu)>0$ for all $\mu\geq0$, then $\gamma=0.$
\end{lem}

From the Lemma \ref{lem4.3}, we need to calculate the index of $F(d_1,d_2;.)$ at $\tilde{\mathbf{u}}$. The crucial step is to determine the range of $\mu$ for which $H_i(d_1,d_2;\mu)<0$.  By using the same method as in \cite{MLW,PSW1,PSW}, we get the following theorem for the non-constant steady state solutions.

\begin{thm} \label{nonconstant}
	Assume that $M/d_1\in(\mu_j,\mu_{j+1})$, for some $j\geq 1$, and $\gamma_{j}=\sum_{i=1}^j m(\mu_i)$ is odd, then there exists a positive constant $d_{2}^*$, such that for any $d_2\geq d_{2}^*$, the problem of \eqref{2.3} has at least one nonconstant positive solution.
\end{thm}

\begin{proof}
	Since $M>0$, we can find a $
	\hat{d}_2\gg 1$ such that \eqref{4.13} holds and $0<\mu_{-}(d_1,d_2)<\mu_{+}(d_1,d_2)$.
	
	According to $M/d_1\in (\mu_j,\mu_{j+1})$ and \eqref{4.14}, there exist $d_0>\hat{d}_2$ such that
	\be\label{4.15}
	\mu_{+}(d_1,d_2)\in(\mu_{j},\mu_{j+1}),\quad  0<\mu_{-}(d_1, d_2)<\mu_1,\quad\forall\ d_2\geq d_0.
	\ee
	By Theorem \ref{thm3}, there exists $d_{1}^*>d_0$, such that \eqref{2.3} has no nonconstant positive solution for $d_1=d_{1}^*$ and $d_2\geq d_{1}^*$.  Moreover, we can choose $d_{1}^*$ large enough that $M/d_{1}^*<\mu_1$.  Applying \eqref{4.14} once again, there exists a constant $d_{2}^*>d_{1}^*$, such that
	\be\label{4.16}
	0<\mu_{-}(d_{1}^*,d_{2}^*)<\mu_{+}(d_{1}^*,d_{2}^*)<\mu_1, \quad\forall\ d_2\geq d_{2}^*.	
	\ee
	
	We shall prove that for any $d_2\geq d_{2}^*$, \eqref{2.3} has at least one nonconstant positive solution.  On the contrary, suppose that \eqref{2.3} has no nonconstant positive solution for some $d_2\geq d_{2}^*$.  In the following, we will derive a contradiction by using a homotopy argument.
	
	For these fixed $d_{1}^*,d_{2}^*,d_1$ and $d_2$, we define
	\bess
	D(t) =
	\begin{pmatrix}
		td_1+(1-t)d_{1}^*&\quad 0\\[.5mm]
		0 &\quad td_2+(1-t)d_{2}^*
	\end{pmatrix},\quad 0\leq t\leq 1,
	\eess
	and consider the problem
	\be\label{4.17}
	\begin{cases}
		-\lap\textbf{u}= D^{-1}(t)G(\textbf{u}),&x\in\Omega,\\[.1mm]
		\dd\frac{\pt\textbf{u}}{\pt\nu}=0,&x\in\pt\Omega.
	\end{cases}
	\ee
	Note that $\textbf{u}$ is a nonconstant positive solution of \eqref{2.3} if and only if it is a solution of \eqref{4.17} for $t=1$.  It is obvious that $\tilde{\textbf{u}}$ is the unique positive constant solution of \eqref{4.17}. And $\textbf{u}$ is a nonconstant positive solution of \eqref{4.17} if and only if it is a solution of the problem,
	\be
	\Psi(\textbf{u};t)=\textbf{u}-(I-\lap)^{-1}\{D^{-1}(t)G(\textbf{u})+ \textbf{u}\}=0\qquad\mbox{on}\ \mathbf{X},
	\ee
	where $\mathbf{X}$ is defined by \eqref{4.3}.  It is obvious that
	\[\Psi(\textbf{u};1)=F(d_1,d_2;\textbf{u}),\qquad\Psi (\textbf{u};0)=F(d_{1}^*,d_{2}^*;\textbf{u}),\]
	and
	\bess
	D_{\textbf{u}}F(d_1,d_2;\tilde{\textbf{u}})=I-(I-\lap)^{-1}\{D^{-1} G_{\textbf{u}}(\tilde{\textbf{u}})+I\}=0,\\[1mm]
	D_{\textbf{u}}F(d_{1}^*,d_{2}^*;\tilde{\textbf{u}})=I-(I-\lap)^{-1}\{\tilde{D}^{-1}G_{\textbf{u}}(\tilde{\textbf{u}})+I\}=0,
	\eess
	where $F(d_1,d_2;\tilde{\textbf{u}})$ is defined by \eqref{4.12} and $D=\mbox{diag}(d_1,d_2),\tilde{D}=\mbox{diag}(d_{1}^*,d_{2}^*)$.  The above arguments show that both $\Psi(\textbf{u};1)=0$ and $\Psi(\textbf{u};0)=0$ have no nonconstant positive solutions.
	
	By \eqref{4.15} and \eqref{4.16}, it follows that
	\[\Lambda (d_1,d_2)\cap S_p=\{\mu_1,\mu_2,\cdots,\mu_j\},\quad\Lambda (d_{1}^*,d_{2}^*)\cap S_p=\emptyset.\]
	Since $\gamma_j$ is odd, from Lemma \ref{lem4.3},
	\be\label{4.19}
	\begin{cases}
		\mbox{index}(\Psi(\cdot\ ;1),\tilde{\textbf{u}})=\mbox{index}(F(d_1, d_2; \cdot\ ), \tilde{\textbf{u}})=(-1)^{\gamma_j}=-1,\\[1mm]
		\mbox{index}(\Psi(\cdot\ ;0),\tilde{\textbf{u}})=\mbox{index}(F(d_1, d_2; \cdot\ ), \tilde{\textbf{u}})=(-1)^{0}=1.
	\end{cases}
	\ee
	By using Theorem \ref{thm4.1}, there exists a positive constant $\underline{C}$ and $\overline{C}$ such that, for all $0\leq t\leq 1$, the positive solution $(u,v)$ of \eqref{4.17} satisfies $\underline{C}\leq u(x),v(x)\leq\overline{C}$ on $\overline\Omega $. Set
	\[\Sigma=\{\textbf{u}\in\mathbf{X}:\, \underline{C}<u(x),v(x)<\overline{C}\ \mbox{on}\ \overline\Omega \}.\]
	Furthermore, $\Psi(\textbf{u};t)\neq 0$ for all $\textbf{u}\in\pt\Sigma$ and $0\leq t\leq 1$.  By the homotopy invariance of the Leray-Schauder degree \cite{N},
	\be\label{4.20}
	\mbox{deg}(\Psi(\cdot\ ;0),\Sigma,0)=\mbox{deg}(\Psi(\cdot\ ;1), \Sigma,0).
	\ee
	Since both equations $\Psi(\textbf{u};0)=0$ and $\Psi(\textbf{u};1)=0$ have the unique positive solution $\tilde{\textbf{u}}$ in $\Sigma$, by \eqref{4.19}, we obtain
	\bess
	\mbox{deg}(\Psi(\cdot\ ;0),\Sigma,0)=\mbox{index}(\Psi(\cdot\ ; 0),\tilde{\textbf{u}})=1,\\[1mm]
	\mbox{deg}(\Psi(\cdot\ ;1),\Sigma,0)=\mbox{index}(\Psi(\cdot\ ;1),\tilde{\textbf{u}})=-1,	
	\eess
	this contradicts \eqref{4.20} and the proof is complete.
\end{proof}
\section{Conclusions}
Fear is a psychological feature of most animals. In ecology, almost all prey populations are concerned about potential predator attacks. Apart from direct consumption, predator populations often cause fear in prey populations and promote a variety of avoidance mechanisms (morphological and/ or behavioral). Sometimes it affects the life history of prey populations and dramatically reduces reproductive rates.  In this paper, we introduced a diffusive prey-predator model that incorporates the effect that fear of predators has on prey with the Beddington-DeAngelis functional response under homogeneous Neumann boundary conditions. Biologically, the model highlights how both direct predation and fear-induced behavioural changes can significantly influence prey population dynamics and spatial patterns. Firstly, we investigated \eqref{2.2} has a unique positive constant equilibrium solution under the conditions of $cp>mq, r>d+a\lambda$, and discuss the global existence, uniqueness, and estimates of the solution of \eqref{2.2}.  Secondly, we give a priori estimates to discuss the nonexistence and existence of nonconstant positive solutions of \eqref{2.3} by using the maximum principle. And we show the nonexistence of nonnegative steady-state solutions when the diffusion coefficients $d_1$ and $d_2$ are large (Theorem \ref{thm4.1}). We also investigate the nonconstant positive steady state to \eqref{2.3} under the homogeneous Neumann boundary condition. By applying the Leray–Schauder degree theory, we provide sufficient conditions for the existence of a nonconstant positive solution. In addition, we give some numerical simulations to verify Theorem \ref{nonconstant} and complement our theoretical analysis results. The following Figure \ref{fig1} and Figure \ref{fig2} show the time evolution of prey and predator densities under varying diffusion and interaction strengths. For each case, the system is initialized near a spatially homogeneous steady state. Pattern formation emerges due to bifurcation driven by diffusion imbalance and nonlinear predator-prey interactions. The sharp features in predator density in Figure \ref{fig2} reflect localized, nonlinear predator response under low diffusion. 

\begin{figure}[H]
\centering
\includegraphics[scale=0.35]{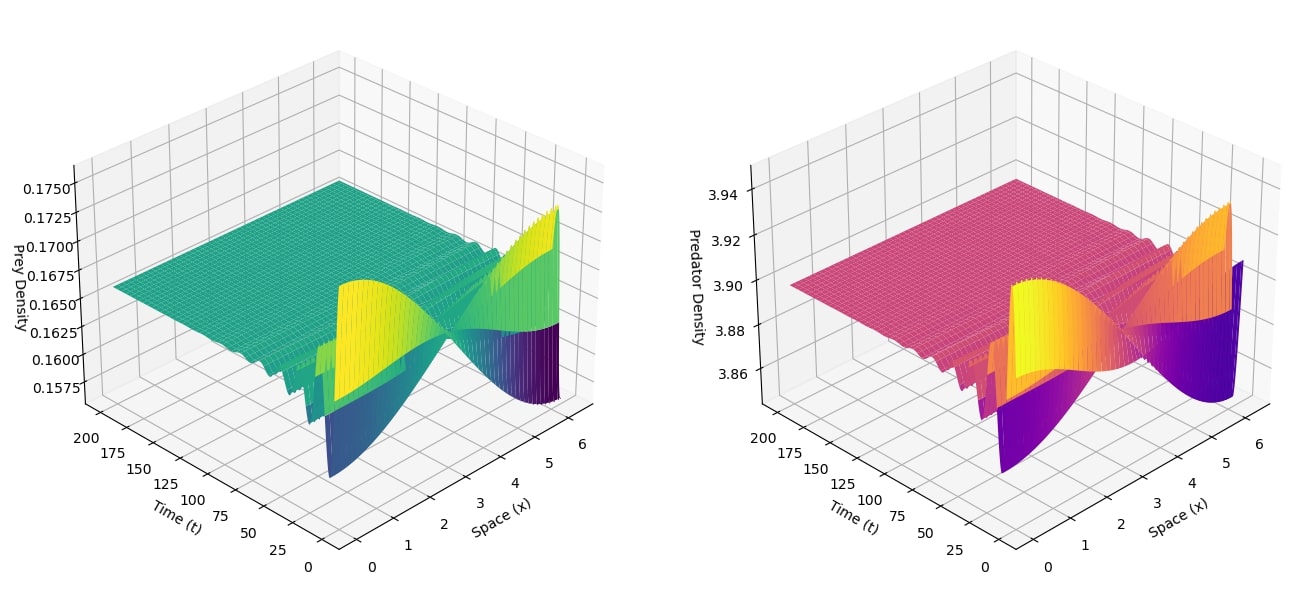}
\caption{Nonconstant positive solution $(u(x,t),v(x,t))$ of the system \eqref{2.3}, $t=200, d_{2}=0.1, a=0.1, \tilde{u}=0.16608, \tilde{v}=3.89934,$ initial data $(u_{0},v_{0})=(\tilde{u}+0.01\cos \frac{x}{2},  \tilde{v}+0.01\cos x,)$ $\frac{M}{d_{1}}=2.407\in (\mu_{1},\mu_{2})=(1,4)$.}
\label{fig1}
\end{figure}

\begin{figure}[H]
\centering
\includegraphics[scale=0.35]{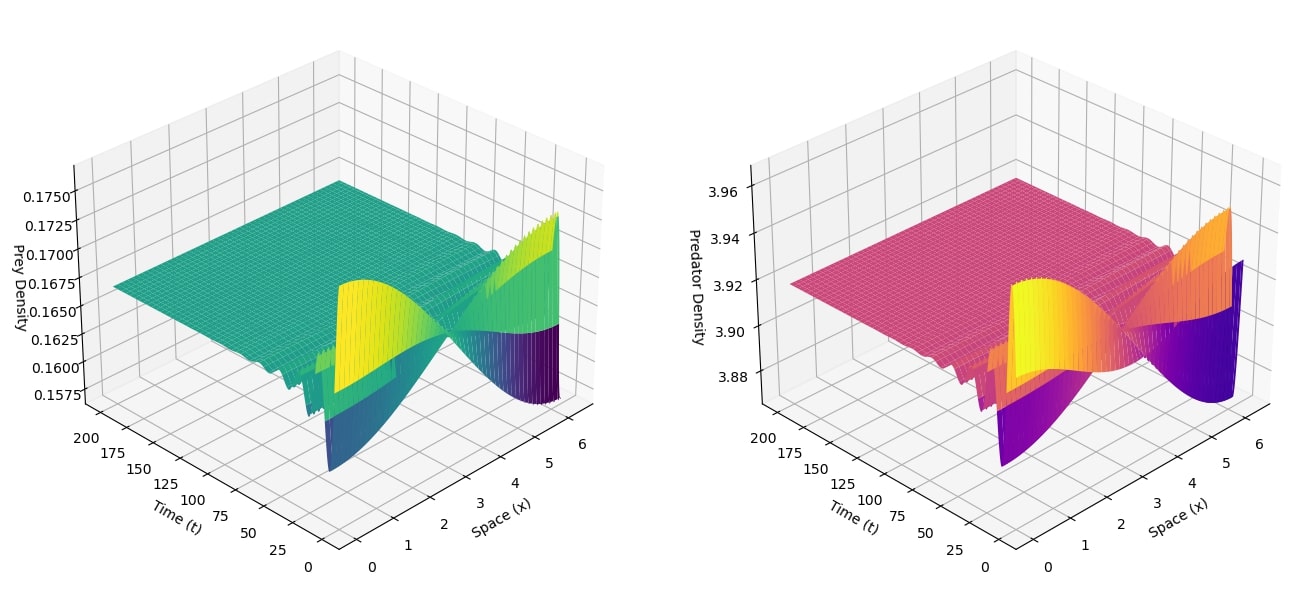}
\caption{Nonconstant positive solution $(u(x,t),v(x,t))$ of the system \eqref{2.3}, $t=200, d_{2}=0.2, a=0.055, \tilde{u}=0.16675, \tilde{v}=3.91904,$ initial data $(u_{0},v_{0})=(\tilde{u}+0.01\cos \frac{x}{2},  \tilde{v}+0.01\cos x,)$ $\frac{M}{d_{1}}=10.095\in (9,16)$.}
\label{fig2}
\end{figure}

Note that the corresponding bifurcation $\gamma_{j}$ is odd for the eigenvalue intervals, ensuring the bifurcation leads to the nonconstant positive solution shown. 

\begin{figure}[H]
\centering
\includegraphics[scale=0.25]{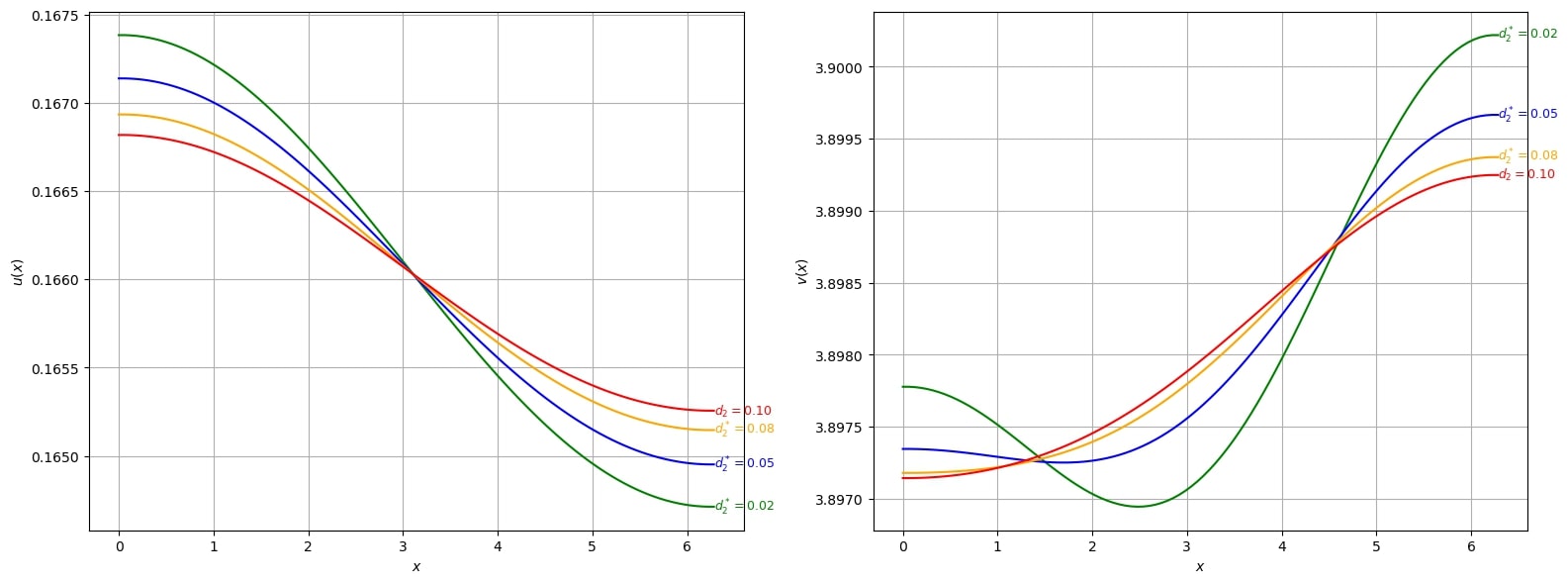}
\caption{Nonconstant positive solution $(u(x,t),v(x,t))$ of the system \eqref{2.3}, $ a=0.1, \tilde{u}=0.16608, \tilde{v}=3.89934,$ initial data $(u_{0},v_{0})=(\tilde{u}+0.01\cos \frac{x}{2},  \tilde{v}+0.01\cos x), t=50, d_{2}=0.02, 0.05, 0.08$ and $0.1$.}
\label{fig3}
\end{figure}

\begin{figure}[H]
\centering
\includegraphics[scale=0.29]{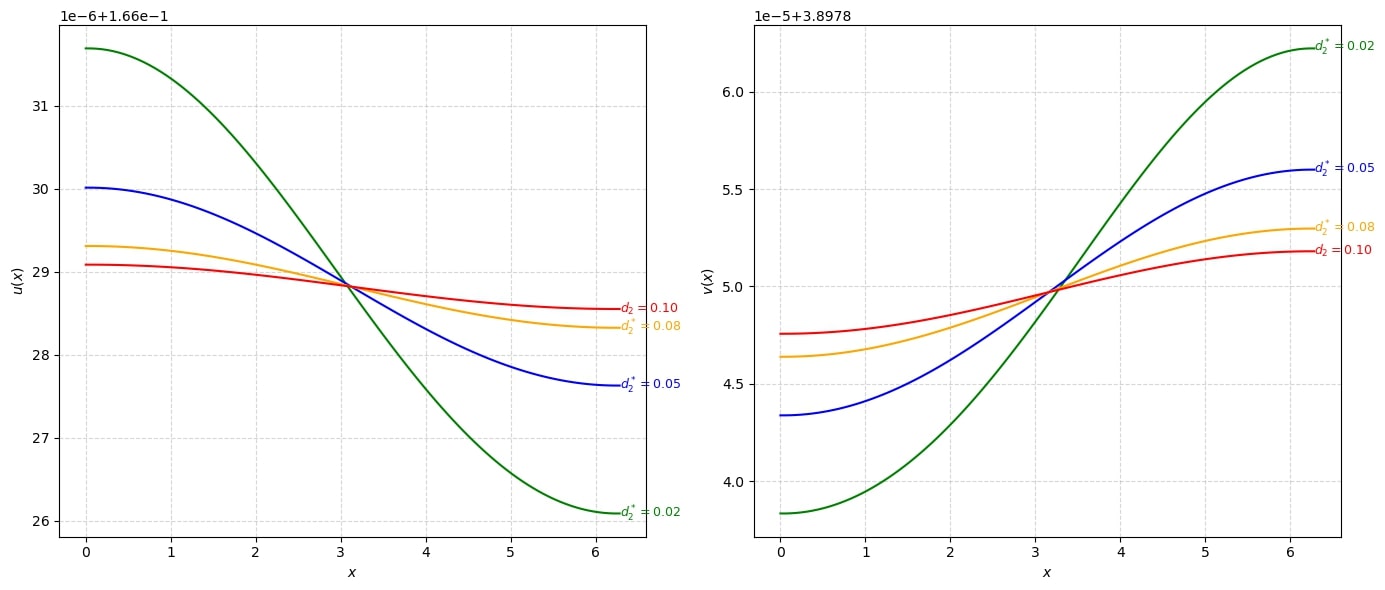}
\caption{Nonconstant positive solution $(u(x,t),v(x,t))$ of the system \eqref{2.3}, $a=0.1, \tilde{u}=0.16608, \tilde{v}=3.89934,$ initial data $(u_{0},v_{0})=(\tilde{u}+0.01\cos \frac{x}{2},  \tilde{v}+0.01\cos x), t=200, d_{2}=0.02, 0.05, 0.08$ and $0.1$. }
\label{fig4}
\end{figure}

\begin{figure}[H]
\centering
\includegraphics[scale=0.25]{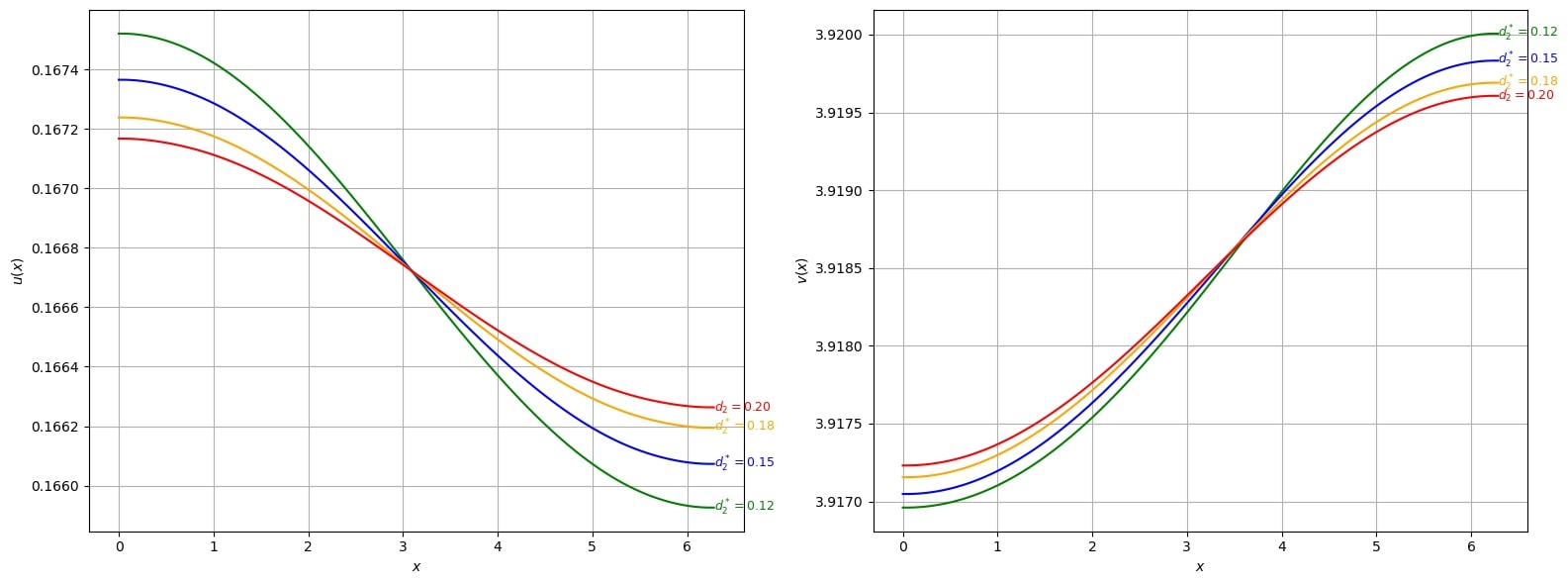}
\caption{ Nonconstant positive solution $(u(x,t),v(x,t))$ of the system \eqref{2.3}, $a=0.055, \tilde{u}=0.16675, \tilde{v}=3.91904,$ initial data $(u_{0},v_{0})=(\tilde{u}+0.01\mathrm{cos}\frac{x}{2},  \tilde{v}+0.01\mathrm{cos} x),t=50, d_{2}=0.12, 0.15, 0.18$ and $0.2$.}
\label{fig5}
\end{figure}

\begin{figure}[H]
\centering
\includegraphics[scale=0.27]{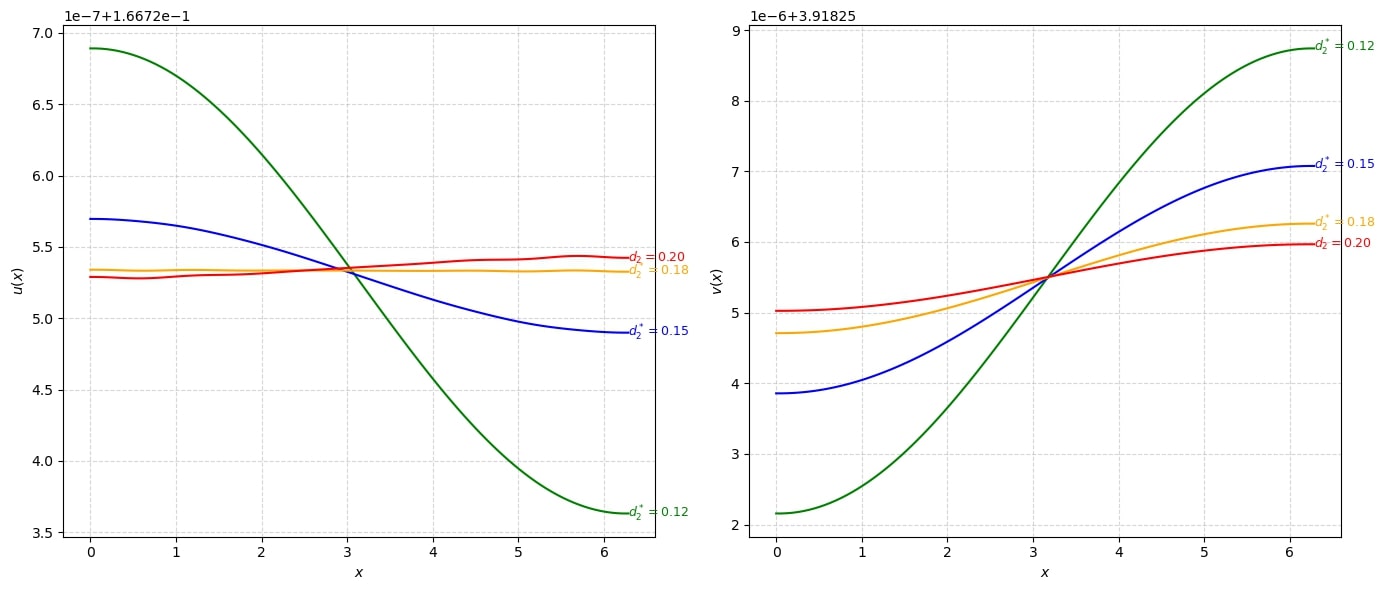}
\caption{Nonconstant positive solution $(u(x,t),v(x,t))$ of the system \eqref{2.3}, $ a=0.055, \tilde{u}=0.16675, \tilde{v}=3.91904,$ initial data $(u_{0},v_{0})=(\tilde{u}+0.01\cos \frac{x}{2},  \tilde{v}+0.01\cos x), t=200, d_{2}=0.12, 0.15, 0.18$ and $0.2$}
\label{fig6}
\end{figure}

The Figure \ref{fig3} - Figure \ref{fig6} illustrate the long-time behavior of the predator–prey reaction-diffusion system under different values of the predator diffusion rate $d_{2},$ with prey diffusion $d_{1}$ fixed. 

\section*{Acknowledgment} 
This work was supported by SIIT Young Researcher Grant, under a contract
number SIIT 2022-YRG-AS01. A crucial part of the research was carried out while A.Z.M.  was visiting Sirindhorn International Institute of Technology (SIIT) at Thammasat University, and its hospitality is gratefully appreciated.  A.C.M. and T.T.S. gratefully acknowledge financial support from the Excellent Foreign Student (EFS) scholarship,  Sirindhorn International Institute of Technology (SIIT),  Thammasat University. M.H.L. was supported by the Faculty-Quota Scholarship, Sirindhorn International Institute of Technology (SIIT),  Thammasat University.


\end{document}